\documentclass[a4paper,11pt,reqno]{amsart}

\usepackage[T1]{fontenc}
\usepackage[utf8]{inputenc}
\usepackage{fullpage}

\usepackage{amsmath, amsthm, amssymb, mathrsfs, amsfonts}
\usepackage{array}
\usepackage[active]{srcltx}
\usepackage{graphicx}
\usepackage{float}
\usepackage{setspace}
\usepackage{longtable}
\usepackage[margin=3cm]{geometry}
\usepackage{cite}
\usepackage[colorlinks=true,allcolors=blue]{hyperref}
\usepackage{enumerate}
\usepackage[square,numbers]{natbib}
\usepackage{comment}
 \usepackage{systeme}

\newtheorem{theor}{Theorem}[section]
\newtheorem{examp}{Example}[section]

\newtheorem{prop}{Proposition}[section]
\newtheorem{cor}{Corollary}[section]

\theoremstyle{definition}

\newtheorem{rem}{Remark}

\def\car{\mathop{\rm char}\nolimits}

\usepackage{tikz,xcolor,hyperref}

\definecolor{lime}{HTML}{A6CE39}
\DeclareRobustCommand{\orcidicon}{%
	\begin{tikzpicture}
	\draw[lime, fill=lime] (0,0)
	circle [radius=0.16]
	node[white] {{\fontfamily{qag}\selectfont \tiny ID}};
	\draw[white, fill=white] (-0.0625,0.095)
	circle [radius=0.007];
	\end{tikzpicture}
	\hspace{-2mm}
}

\foreach \x in {A, ..., Z}{%
	\expandafter\xdef\csname orcid\x\endcsname{\noexpand\href{https://orcid.org/\csname orcidauthor\x\endcsname}{\noexpand\orcidicon}}
}

\title[Cayley unitary elements \& oriented involutions]%
      {Cayley unitary elements in group algebras \\ under oriented involutions}

\author[J. H. Castillo]{John H. Castillo\orcidC{}}
\address{John H. Castillo, Departamento de Matem\'aticas y Estad\'istica, Universidad de Nari\~no}
\email{jhcastillo@udenar.edu.co}

\author[Y. W. G\'omez-Esp\'indola]{Yzel Wlly G\'omez-Esp\'indola\orcidB{}}
\address{Yzel Wlly G\'omez-Esp\'indola, Escuela de Matem\'aticas, Universidad Industrial de Santander, Santander, Colombia}
\email{wage03@gmail.com}

\author[A. Holgu\'in-Villa]{Alexander Holgu\'in-Villa\orcidA{}}
\address{Alexander Holgu\'in-Villa, Escuela de Matem\'aticas, Universidad Industrial de Santander, Santander, Colombia}
\email{aholguin@uis.edu.co}

\keywords{Group algebras; Skew elements; Cayley unitary elements; Oriented involutions.}
\subjclass[2010]{16U60, 16W10, 16S34, 16R50.}

\begin{document}
\maketitle
 \noindent
  \renewcommand{\refname}{References}


\begin{abstract}
   Let $\mathbf{F}$ be a real extension of $\mathbb{Q}$, $G$ a finite group and $\mathbf{F}G$ its group algebra.
   Given both a group homomorphism $\sigma:G\rightarrow \{\pm1\}$ (called an orientation) and a group involution
   $^\ast:G \rightarrow G$ such that $gg^\ast\in N=ker(\sigma)$, an oriented group involution $\circledast$ of
   $\mathbf{F}G$ is defined by  $\alpha=\sum_{g\in G}\alpha_{g}g \mapsto \alpha^\circledast=\sum_{g\in G}\alpha_{g}\sigma(g)g^{\ast}$.
   In this paper, in case the involution on $G$ is the classical one, $x\mapsto x^{-1}$, $\beta=x+x^{-1}$ is a skew-symmetric element in $\mathbf{F}G$ such that $1+\beta$ is invertible, for $x\in G$ with $\sigma(x)=-1$, we consider Cayley unitary elements built out of $\beta$. We prove that the coefficients of $(1+\beta)^{-1}$ involve
   an interesting sequence which is a Fibonacci-like sequence.

  \end{abstract}

\section{Introduction}

Let $\mathbf{F}G$ denote the group algebra of the group $G$ over the field $\mathbf{F}$ with $\car(\mathbf{F})\neq 2$.
Any involution $^\ast: G \rightarrow G$ can be extended $\mathbf{F}$-linearly to an algebra involution of $\mathbf{F}G$ in
an obvious way: if $\alpha=\sum_{g\in G}\alpha_{g}g\in \mathbf{F}G$, then $\alpha^{\ast}=\sum_{g\in G}\alpha_{g}g^{\ast}$.
Such a map is called a group involution of $\mathbf{F}G$. A natural example is the so-called \emph{classical
involution}, which is induced from the map $g\mapsto g^\ast=g^{-1}$, $g\in G$.

Let $\sigma:G\rightarrow \{\pm 1\}$ be a non-trivial homomorphism (called an \emph{orientation} of $G$). If
$^\ast:G \rightarrow G$ is a group involution, an {\it oriented group involution} of $\mathbf{F}G$ is defined by

\begin{equation}\label{eq0}
\alpha=\sum_{g\in G} \alpha_gg \mapsto \alpha^\circledast=\sum_{g\in G} \alpha_g\sigma(g)g^{\ast}.
\end{equation}

Notice that, as $\sigma$ is non-trivial, $\car(\mathbf{F})$ must be different from $2$. It is clear that,
$\alpha\mapsto \alpha^\circledast$ is an involution in $\mathbf{\mathbf{F}}G$ if and only if $gg^\ast\in N=ker(\sigma)$
for all $g\in G$.

In the case when the involution on $G$ is the classical involution, the map $\circledast$ is precisely
the oriented involution introduced by S.~P. Novikov, \cite{Nov:70}, in the context of $K$-theory.

We denote by $\mathbf{F}G^{-}=\{\beta\in \mathbf{F}G:\beta^\circledast=-\beta\}$ the Lie subalgebra of skew-symmetric
elements of $\mathbf{F}G$ under the usual Lie bracket, $[\alpha, \gamma]=\alpha\gamma-\gamma\alpha$. It is not hard to see
that $\mathbf{F}G^{-}$, as an $\mathbf{F}$-module, is spanned by the set
\begin{equation}\label{eq01}
\mathcal{L}=\{g-g^{\ast}: g^{\ast}\ne g, g\in N\}\cup \{g: g^{\ast}=g, g\notin N\}\cup \{g+g^{\ast}: g^{\ast}\neq g, g\notin N\}.
\end{equation}
Writing $\mathcal{U}(\mathbf{F}G)$ for the group of units of $\mathbf{F}G$, we denote by $\mathcal{U}^+(\mathbf{F}G)$ the set
of symmetric units of $\mathbf{F}G$, namely, those units fixed by $\circledast$ and $\mathcal{U}n(\mathbf{F}G)$ the subgroup
of unitary units, i.e., $\mathcal{U}n(\mathbf{F}G)=\{u\in \mathcal{U}(\mathbf{F}G): uu^{\circledast}=1\}$. Note that if
$\beta\in \mathbf{F}G^{-}$ then $1+\beta$ is invertible in $\mathbf{F}G$ if and only if $1-\beta$ is invertible too. Furthermore,
if $\beta\in \mathbf{F}G$ is a skew-symmetric element such that $1+\beta$ is invertible, $u_{[\beta]}=(1-\beta)(1+\beta)^{-1}$
belongs to $\mathcal{U}n(\mathbf{F}G)$ and is called a {\it Cayley unitary element built out of $\beta$}. We let
$\mathcal{U}n^{C}(\mathbf{F}G)$ denote the set of all Cayley unitary elements of $\mathbf{F}G$. Besides, it is easy to
check that $u^{-1}_{[\beta]}=u_{[-\beta]}$.

An interesting problem in the study of group algebras is that of relating properties of the group of units $\mathcal{U}(\mathbf{F}G)$
of $\mathbf{F}G$ to properties of the group algebra or more specifically to the structure of the group $G$. In the late 70's B. Hartley
conjectured that if $G$ is a torsion group and $\mathcal{U}(\mathbf{F}G)$ satisfies a group identity, then $\mathbf{F}G$ must satisfy
a polynomial identity. The answer to this conjecture lead to obtain a complete characterization of group algebras $\mathbf{F}G$ for which
$\mathcal{U}(\mathbf{F}G)$ satisfies a group identity. However, particular group identities are also of considerable interest. For instance,
the conditions under which $\mathcal{U}(\mathbf{F}G)$ is nilpotent, $n$-Engel for some $n$, bounded Engel and solvable also has been
studied extensively. For further details about these results, see G. Lee \cite{Lee:10} and the references quoted therein.

When we consider an involution $\ast$ on the group $G$ extended $\mathbf{F}$-linearly to $\mathbf{F}G$, also denoted by $\ast$, one can
prove that some special subsets of the unit group $\mathcal{U}(\mathbf{F}G)$ constructed with this involution and satisfying a group
identity determine either the structure of the whole group units $\mathcal{U}(\mathbf{F}G)$ or the whole group algebra $\mathbf{F}G$.
For instance, when $\ast$ is the classical involution induced from $g\mapsto g^{-1}, g\in G$, was confirmed a stronger version of the
Hartley's conjecture by showing that, if $G$ is a torsion group, $\mathbf{F}$ is a infinite field with $\car(\mathbf{F})\neq 2$ and
$\mathcal{U}^+(\mathbf{F}G)$ satisfies a group identity, then $\mathbf{F}G$ satisfies a polynomial identity. It should be noted that, in
this case, specific group identities such as nilpotency and solvability, have also received a good deal of attention, and again we refer
the reader to \cite{Lee:10} for an overview. Also, in recent years, identities satisfied by the set of symmetric units with respect to
other involutions have been examined; see A. Dooms and M. Ruiz Mar\'in \cite{DRM:07}, A. Giambruno, C. Polcino Milies and S. K. Sehgal
\cite{GPS:09i, GPS:10}, A. Holguín-Villa and J. H. Castillo \cite{holguin2023group} and G. Lee, S. K. Sehgal and E. Spinelli \cite{LSS10}. A similar question may be posed for the subgroup of unitary
units, i.e., to determine the extent to which the properties of $\mathcal{U}n(\mathbf{F}G)$ determine either properties of $\mathcal{U}(\mathbf{F}G)$
or $\mathbf{F}G$. Only a few articles have examined identities satisfied by the unitary units. For instance, J. Gon\c calves and D. Passman
\cite{GonPass:01} discussed when $\mathcal{U}n(\mathbf{F}G)$ contains a non-abelian free group. A. Giambruno and C. Polcino Milies \cite{GP:03}
explored group identities satisfied by $\mathcal{U}n(\mathbf{F}G)$, for semiprime group algebras $\mathbf{F}G$. Recently, G. Lee, S. K. Sehgal
and E. Spinelli \cite{LSS14, LSS18} characterized, with mild restrictions, the groups $G$ such that $\mathcal{U}n(\mathbf{F}G)$ is nilpotent;
which in many cases is usually enough to imply that $\mathcal{U}(\mathbf{F}G)$ is nilpotent, but there are exceptions. Otherwise, by allowing
group involutions on $\mathbf{F}G$, O. Broche, A. Dooms and M. Ruiz Mar\'in \cite{BDR:09} gave several partial results.

It is well known that the Cayley unitary elements play an important role in the ring $M_n(D)$ of $n\times n$ matrices over a division ring $D$ with
$\car(D)\neq 2$. For instance, in \cite{ChuangLee:95} C. L. Chuang and P. H. Lee showed that if the involution defined in $M_n(D)$ is non-identity
on $D$, a unitary unit in $M_n(D)$ can be written as a product of two Cayley unitary elements. Now, in the setting of the group algebras and when
$\ast$ is the classical one, Gon\c calves and Passman in \cite{GonPass:01} classified all finite groups $G$ such that $\mathcal{U}n(\mathbf{F}G)$
does not contain a free group of rank $2$, provided the field $\mathbf{F}$ is non-absolute of characteristic not $2$, i.e., either $\car\big(\mathbf{F}\big)=0$
or $\car\big(\mathbf{F}\big)>0$ and $\mathbf{F}$ contains an element transcendental over its prime subfield. Moreover, they built a pair of unitary
units in $\mathbf{F}G$ from Cayley unitary elements and applied their results to verify that the two elements essentially generate a free group. In
\cite{GP:03} Giambruno and Polcino Milies, in case $\car(\mathbf{F})=0$, classified the torsion groups $G$ for which $\mathcal{U}n(\mathbf{F}G)$ does
not contain a free group of rank $2$, provided that $\mathbf{F}G^{-}$ is Lie nilpotent. The relation between the existence of free groups in
$\mathcal{U}n(\mathbf{F}G)$ and the Lie nilpotence of $\mathbf{F}G^{-}$ should not be surprising if we consider the general linear group. Indeed,
Chuang and Lee showed in \cite{ChuangLee:95} that the Cayley unitary elements, which are built out of skew-symmetric elements, fill in the orthogonal
group or a subgroup of the symplectic group of index $2$.

On the other hand, A. C. Vieira and V. Ribeiro in \cite{VieiraRibeiro:06}, in case $\mathbf{F}$ is a real extension of $\mathbb{Q}$ considered
Cayley unitary elements built out of skew-symmetric elements $\beta=q(x-x^{-1})$ in $\mathbf{F}G$ such that $1+\beta$ is invertible
in $\mathbf{F}G$, for $q\in \mathbf{F}$ and $x\in G$. The constructions involve an interesting sequence in the coefficients of
$(1 + \beta)^{-1}$ which is the Fibonacci sequence when $q=1$.

\section{Cayley Unitary Elements and classical involution}

Let $\mathbf{F}$ be a field of characteristic different from $2$ and let $\mathbf{F}G$ be the group algebra of a finite group $G$
over $\mathbf{F}$. Denote by $^\ast : \mathbf{F}G\rightarrow  \mathbf{F}G$ the natural $\mathbf{F}$-involution of $\mathbf{F}G$
induced by setting $x^{\ast} = x^{-1}$ for all $x\in G$. By expression \eqref{eq01}, it is clear that the Lie subalgebra of skew-symmetric
elements defined by $\mathbf{F}G^{-}=\{\beta\in \mathbf{F}G: \beta^{\ast}=-\beta\}$, as a $\mathbf{F}$-module, is generated by the
elements $x-x^{-1}$, for $x\in G$. Furthermore, when $\beta$ is a skew-symmetric element in $\mathbf{F}G$ such that $1+\beta$ is invertible,
the element $u_{[\beta]}=(1-\beta)(1+\beta)^{-1}$ belongs to $\mathcal{U}n(\mathbf{F}G)$ and is called a Cayley unitary element built out of
$\beta$. We let $\mathcal{U}n^{\mathbf{C}}(\mathbf{F}G)$ denote the set of all Cayley unitary elements of $\mathbf{F}G$.

\subsection{Some elementary properties}

The Cayley unitary elements satisfy interesting properties. For example, since that $u_{[\beta]}u_{[-\beta]}=1=u_{[-\beta]}u_{[\beta]}$, thus $u_{[\beta]}^{-1}=u_{[-\beta]}$, for $\beta\in \mathbf{F}G^{-}$. The following result allow us to determine when a unitary element is a Cayley unitary element and when it
is a product of two Cayley unitary elements, respectively.

\begin{prop}\label{Propo-cay.sii.inver}
Let $R$ be a ring with involution $\ast$ in which $2$ is invertible. Then the following properties hold:
\begin{enumerate}
\item \emph{(\cite[Lemma 1]{ChuangLee:95})} A unitary element $u$ is a Cayley unitary element if and only if $1+u$
      is invertible in $R$.

\item \emph{(\cite[Lemma 2]{ChuangLee:95})} A unitary element $u$ is a product of two Cayley unitary elements if
      and only if $(1+u)-(1-u)k$ is invertible in $R$, for some skew-symmetric element $k$, with $1+k$ invertible in $R$.

\end{enumerate}
\end{prop}

When the involution on $G$ is the classical one, it is clear that $G\subset \mathcal{U}n(\mathbf{F}G)$ but in general
$G\not\subset \mathcal{U}n^{C}(\mathbf{F}G)$. Let $x\in G$ be of order $n$, then

\begin{equation}\label{equ-1+x unidad}
(1+x)\left[\dfrac{1}{2} \left(1-x+x^{2} - \cdots + (-1)^{n-1} x^{n-1}\right)\right]=\dfrac{1}{2}\big(1+(-1)^{n-1}\big).
\end{equation}

By Proposition \ref{Propo-cay.sii.inver}, it is clear that $x\in \mathcal{U}n^{C}(\mathbf{F}G)$ if and only if $n$ is
odd. It follows that for a finite group $G$ we have  $G\subset \mathcal{U}n^{C}(\mathbf{F}G)$ if and only if $G$ has odd order.
Moreover, if $x\in G$ has odd order $n>1$, $x=u_{[\beta]}$ for some $\beta\in \mathbf{F}G^{-}$ and $1+x=2(1+\beta)^{-1}$.
By expression \eqref{equ-1+x unidad} we have $1+x=2\lbrack 1-(x-x^{n-1})-(x^3-x^{n-3})-\cdots -(x^{n-2}-x^2)\rbrack^{-1}$.
Therefore, we have the next result, see \cite[Corollary 1]{VieiraRibeiro:06}.

\begin{cor}\label{cor1}
Let $G$ be a finite group, $\mathbf{F}$ a field with $\car(\mathbf{F})\neq 2$ and $\mathbf{F}G$ its group algebra. Then
$G\subset \mathcal{U}n^{C}(\mathbf{F}G)$ if and only if $G$ has odd order. In this case, if $x$ has order $n>1$, $x=u_{[\beta]}$
where $\beta=-(x-x^{n-1})-(x^3-x^{n-3})-\cdots -(x^{n-2}-x^2)$.
\end{cor}

\subsection{Cayley unitary elements built out of skew-symmetric elements}

Let $\mathbf{F}$ be a real extension of $\mathbb{Q}$. In \cite{VieiraRibeiro:06}, Vieira and Ribeiro gave a way to find
Cayley unitary elements built out of generator $q(x-x^{-1})$ of $\mathbf{F}G^{-}$, with $q\in \mathbf{F}$ and $x\in G$,
such that $1+q(x-x^{-1})$ is invertible in $\mathbf{F}G$. In fact, they considered, for $0\neq q\in \mathbf{F}$,
the sequence $(G_i)_{i\in \mathbb{N}}$ recursively defined by
\begin{equation}\label{Gi-rec}
G_0=0,  \, \,  G_1 = 1 \quad  \text{ and } \quad  G_i = q^2 G_{i-2} + G_{i-1} \, \, \,  \text{for} \, \,  i\geq 2,
\end{equation}
from which they demonstrated the next result.

\begin{theor}[{\cite[Theorem 1']{VieiraRibeiro:06}}]\label{Teo-inve-alpha}
Let $x\in G$ be an element of order $n>2$ and $q\in \mathbf{F}$, with $\mathbf{F}$ a real extension of $\mathbb{Q}$.
Then the element $1+ q(x-x^{-1})$ is invertible in $\mathbf{F}G$ and its inverse is given by
$a_{0}+a_{1}x+\cdots+a_{n-1}x^{n-1}$, where for $i=0, 1, \ldots, n-1$

\begin{equation}\label{Form-ai-inve-alpha}
  a_{i}=\frac{q^{n-i}G_{i}+(-q)^{i}G_{n-i}}{G_{n+1}+q^{2}G_{n-1}-q^{n}(1+(-1)^{n})} ~\text{ and }~
G_i= \dfrac{1}{2^{i-1}}\sum_{m \text{ odd}} \binom{i}{m}\left(1+4q^2\right)^{\frac{m-1}{2}}.
\end{equation}
\end{theor}

Note that if $q=1$, then $G_i=F_i$,  the $i$-th term of the well-known Fibonacci sequence. Under these conditions, it follows the next result.

\begin{cor}[{\cite[Theorem 1]{VieiraRibeiro:06}}]\label{Coef-Fibo}
Let $x\in G$ be an element of order $n>2$ and $\mathbf{F}$ is a field of characteristic zero. Then $1+x-x^{-1}$ is
invertible in $\mathbf{F}G$ and its inverse is given by $a_{0}+a_{1}x+\cdots+a_{n-1}x^{n-1}$, where
$$
a_{i}=\frac{F_{i}+(-1)^{i}F_{n-i}}{F_{n+1}+F_{n-1}-(1+(-1)^{n})}, \quad \text{ for } \, \, \,  i=0, 1, \ldots, n-1.
$$
\end{cor}

The previous two results show that $1+q(x-x^{-1})$ is always invertible when $q\in \mathbf{F}$, with $\mathbf{F}$
a real extension of $\mathbb{Q}$ and therefore we can find the Cayley unitary element built out of the skew-symmetric
element $q(x-x^{-1})$. The next theorem shows how to do it.

\begin{theor}[{\cite[Theorem 2]{VieiraRibeiro:06}}]\label{Thm-Cayl-alpha}
Let $x\in G$ be an element of order $n>2$, $q\in \mathbf{F}$ and $\beta=q(x-x^{-1})$, where $\mathbf{F}$ is a real
extension of $\mathbb{Q}$. Then the Cayley unitary element $u_{[\beta]}$ is given by $b_{0}+b_{1}x+\cdots+b_{n-1}x^{n-1}$,
where

\begin{equation}
  b_{0}=\dfrac{G_n - 2q^{2}G_{n-1}+q^{n}(1+(-1)^{n})}{G_{n+1}+q^{2}G_{n-1}-q^{n}(1+(-1)^{n})}
  \quad \text{ and } \quad b_{i}=2a_i , \, \, \, \text{ for } \, \, \, i=1, 2, \ldots, n-1,
\end{equation}

with $a_i$'s are as in \eqref{Form-ai-inve-alpha}.
\end{theor}
\section{Cayley Unitary Elements and oriented classical involution}

Let $\mathbf{F}$ be a field with $\car(\mathbf{F})=0$. If we consider on $\mathbf{F}G$ the {\it oriented group involution}
$\sum_{g\in G} \alpha_gg \stackrel{\circledast}\mapsto \sum_{g\in G} \alpha_g\sigma(g)g^{\ast}$, where $G$ is a finite group
with a non-trivial homomorphism $\sigma:G\rightarrow \{\pm1\}$ and an involution $\ast$ and if $N=ker(\sigma)=\{g\in G : \sigma(g)=1\}$
denotes the kernel of $\sigma$, then $N$ is a subgroup in $G$ of index $2$. It is clear that the involution $\circledast$ coincides
on the subalgebra $\mathbf{F}N$ with the group involution $\ast$. Since $\sigma(x)=1$, i.e., $x\in N$, if $x$ has odd order $n>1$,
then by Corollary \ref{cor1} if the involution on $\mathbf{F}G$ is the oriented classical one, it follows that $x=u_{[\beta]}$ where
$\beta=-(x-x^{n-1})-(x^3-x^{n-3})-\cdots -(x^{n-2}-x^2)$. Now, by \eqref{eq01} the set $\mathbf{F}G^{-}$ of skew-symmetric elements
of $\mathbf{F}G$ is generated as an $\mathbf{F}$-module by $\mathcal{L}=\mathcal{L}_1\cup \mathcal{L}_2 \cup \mathcal{L}_3$,
where $\mathcal{L}_1=\{g-g^{-1} : g^2 \neq 1, g\in N \}$, $\mathcal{L}_2=\{g: g^2 =1, g\not\in N\}$, and
$\mathcal{L}_3=\{g+g^{-1} : g^2 \neq 1, g\not\in N \}$.

\begin{rem}\label{construct_cayley}
Notice that to construct Cayley unitary elements from skew-symmetric elements we can consider the last ones as follows. Firstly, for elements in
the $\mathbf{F}$-submodule generated by $\mathcal{L}_1$, the construction is as in the previous section. Secondly, note that if $x\in \mathcal{L}_2$ and
$q\in \mathbf{F}\setminus\{\pm 1\}$ then the inverse of $1+qx$ is equal to $\frac{1}{1-q^2}-\frac{q}{1-q^2}x$ and therefore
$u_{[qx]}=\frac{1+q^2}{1-q^2}-\frac{2q}{1-q^2}x$.
\end{rem}
In the sequel, we study this problem for skew-symmetric elements in $\mathcal{L}_3$. Hence, we are going to consider skew-symmetric elements $\beta=x+x^{-1}$
such that $x\in G$ is an element of even order $n\geq 4$ and $x\not\in N$. Our goal in the next result is to obtain a closed expression for the inverse of
$1+(x+x^{-1})$.

\begin{prop}\label{reg-ai}
Let $\mathbf{F}$ be a field with $\car(\mathbf{F})=0$, $G$ a group with a non-trivial orientation $\sigma$ and $\mathbf{F}G$ its group
algebra endowed with the oriented classical involution. If $x\in G$ is an element of even order $n\geq 4$ with $\sigma(x)=-1$, then
$1+(x+x^{-1})$ is invertible in $\mathbf{F}G$ if and only if $n\equiv 2,4 \pmod{6}$. In these cases, the inverse of $1+(x+x^{-1})$ is
given by $a_0 + a_1 x + a_2x^2 + \cdots + a_{n-1}x^{n-1}$, where
\begin{enumerate}
\item if $n\equiv 2\pmod{6}$, then for $0\leq k \leq n-1$
\[a_k=
\begin{cases}
-\frac{1}{3}, & \text{if $k\equiv 0,2\;(\bmod\;{3})$,}\\
\frac{2}{3}, & \text{if $k\equiv 1\;(\bmod\;{3})$.}
\end{cases}
\]

\item if $n\equiv 4\pmod{6}$, then for $0\leq k \leq n-1$
\[a_k=
\begin{cases}
\frac{1}{3}, & \text{if $k\equiv 0,1\;(\bmod\;{3})$,}\\
-\frac{2}{3}, & \text{if $k\equiv 2\;(\bmod\;{3})$.}
\end{cases}
\]
 \end{enumerate}
\end{prop}
\begin{proof}
To prove that $1+(x+x^{-1})$ is invertible it is enough to verify that
$$
\left(1+\left(x+x^{-1}\right)\right)\left(a_0 + a_1 x + a_2x^2 + \cdots + a_{n-1}x^{n-1}\right)=1.
$$
So
\begin{equation}\label{eq:casos}\begin{cases}
a_{n-1} + a_{0} + a_{1}&=1, \\
a_{k-2} + a_{k-1} + a_{k}&=0,\\
\end{cases}
\end{equation}
 for $2\leq k\leq n$, where $a_t=a_j$ provided that $t\equiv j \pmod{n}$.
Then
\begin{equation}\label{eq:casos2}\begin{cases}
a_0=-(a_{n-2}+a_{n-1}),\\
a_1 =1+a_{n-2}, \\
\end{cases}
\end{equation}
for $2\leq k \leq n-1$. From the last expression in \eqref{eq:casos} we have that  $a_j=a_{j+3}$ for all $j\in \mathbb{N}$ and thus if $m\equiv j \pmod{3}$, we obtain that $a_m=a_j$. As $n-1=(n-4)+3$, then $a_{n-4}=a_{n-1}$ and also $a_{n-1}=1+a_2$.

Now, if $n=6t$ (with $t\in \mathbb{N}$), then $a_{n-4}=a_{3(2t-2)+2}=a_2$. It follows that $a_2=1+a_2$, a contradiction. Thus in this case $1+(x+x^{-1})$ is not invertible.

Secondly, if $n=6t+2$ (with $t\in \mathbb{N})$, then
\begin{align} \label{eq:casos3}
\begin{split}
a_{n-2}=a_{6t}=a_0~\text{ and }~a_{n-1}=a_{6t+1}=a_1.
\end{split}
\end{align}
By \eqref{eq:casos2} and \eqref{eq:casos3}, we obtain that $a_0=-\frac{1}{3}$ and $a_1=\frac{2}{3}$.

Finally, we assume that $n=6t+4$ (with $t\in \mathbb{N})$. Also note that making $k=n-1$ and $k=n$, in the last expression
of \eqref{eq:casos}, we get that $a_{n-3}=a_0$. Then $a_{n-1}=a_{6t+3}=a_0$, $a_{n-3}=a_{6t+1}=a_{1}$. So from \eqref{eq:casos},
we conclude that $a_0=a_1=\frac{1}{3}$. Notice that, in both cases the  expression \eqref{eq:casos} and as has been proven above
$a_m=a_j$ for $m\equiv j \pmod{3}$, we can find the values of $a_k$, for each $2\leq k\leq n-1$.
\end{proof}

The previous result allow us to determine the value of the independent term and the coefficient of the
linear term in the inverse of $1+(x+x^{-1})$, when it exists. From the linear recurrence equation
$a_k=-a_{k-2}-a_{k-1}$ for the sequence $(a_k)_{k\geq 2}$ and $a_0,a_1$ given depending of if $\mathfrak{o}(x)\equiv 2,4\pmod 6$. In the sequel, with techniques of generating functions we will establish a closed form to calculate the terms of the sequence $(a_k)_{k\in \mathbb{N}}$.

By Theorem 5.3.1 in \cite{goodaire2002discrete} we have that the general term of the sequence $(a_k)_{k\in \mathbb{N}}$
is given by
$$
a_k=c_1\left(\dfrac{-1+\sqrt{3}i}{2}\right)^k+c_2\left(\dfrac{-1-\sqrt{3}i}{2}\right)^k,
$$
where $\lambda_1=\dfrac{-1+\sqrt{3}i}{2}$ and $\lambda_2=\dfrac{-1-\sqrt{3}i}{2}$ are the roots of the characteristic polynomial $x^2+x+1$.
Then we have the linear system

\begin{equation}\label{sistema_eq}
\systeme{c_1+c_2=a_0,\lambda_1c_1+\lambda_2c_2=a_1}
\end{equation}

Thus if $\mathfrak{o}(x)\equiv 2\pmod 6$, we obtain that $c_1=\frac{-1-\sqrt{3}i}{6}$ and $c_2=\frac{-1+\sqrt{3}i}{6}$. Then for $k\geq 2$, we have that
\begin{equation}\label{eq:an_general1}
a_k=\dfrac{(-1)^{k-1}}{3(2^{k-1})}\left[ (1-\sqrt{3}i)^{k-1}+(1+\sqrt{3}i)^{k-1}\right].
\end{equation}

Similarly, if $\mathfrak{o}(x)\equiv 4\pmod 6$, we solve the linear system \eqref{sistema_eq} and obtain that $c_1=\frac{1-\sqrt{3}i}{6} $ and $c_2=\frac{1+\sqrt{3}i}{6}$ and then for $k\geq 2$
\begin{equation}\label{eq:an_general2}
a_k=\dfrac{(-1)^{k}}{3(2^{k+1})}\left[ (1-\sqrt{3}i)^{k+1}+(1+\sqrt{3}i)^{k+1}\right].
\end{equation}

The proof of the next result follows immediately from Proposition \ref{reg-ai} and the expressions  \eqref{eq:an_general1} and \eqref{eq:an_general2}.

\begin{theor}\label{THM-inver-si-o(x)=?}
Let $\mathbf{F}$ be a field with $\car(\mathbf{F})=0$, $G$ a group with a non-trivial orientation $\sigma$ and $\mathbf{F}G$ its group algebra endowed
with the oriented classical involution. Let $x\in G$ such that $\sigma (x)=-1$. The following statements are true:
\begin{enumerate}
\item If $\mathfrak{o}(x)=6t$, the element $(1+(x+x^{-1}))$ is not invertible in $\mathbf{F}G$.

\item If $\mathfrak{o}(x)=6t+2$, the element $(1+(x+x^{-1}))$ is invertible in $\mathbf{F}G$ and its
      inverse is in the way $a_0 + a_1 x + a_2x^2 + \cdots + a_{6t+1}x^{6t+1}$, where
      $a_0=-\frac{1}{3}$, $a_1=\frac{2}{3}$ and for all $k\geq 2$
      $$
      a_k=\dfrac{(-1)^{k-1}}{3(2^{k-1})}\left[ (1-\sqrt{3}i)^{k-1}+(1+\sqrt{3}i)^{k-1}\right].
      $$

\item If $\mathfrak{o}(x)=6t+4$, the element $\left(1+\left(x+x^{-1}\right)\right)$
      is invertible in $\mathbf{F}G$ and its inverse is in the way  $a_0 + a_1 x + a_2x^2 + \cdots + a_{6t+3}x^{6t+3}$,
      where $a_0=a_1 =\frac{1}{3}$ and for all $k\geq 2$
      $$
      a_k=\dfrac{(-1)^{k}}{3(2^{k+1})}\left[ (1-\sqrt{3}i)^{k+1}+(1+\sqrt{3}i)^{k+1}\right].
      $$
\end{enumerate}
\end{theor}

\section{Cayley unitary elements constructed from skew-symmetric elements}

In the previous section, we showed how to get $\left(1+\left(x+x^{-1}\right)\right)^{-1}$, when it exists.
This fact allow us to construct Cayley unitary elements from skew-symmetric elements of the type $x+x^{-1}$.

\begin{theor}\label{THM-Cayleyuni-si-o(x)=?}
Let $\mathbf{F}$ be a field with $\car(\mathbf{F})=0$, $G$ a group with a non-trivial orientation $\sigma$ and $\mathbf{F}G$
its group algebra endowed with the oriented classical involution. Let $x\in G$ with $n=\mathfrak{o}(x)\equiv 2,4 \pmod{6}$  and
such that $\sigma (x)=-1$. Then the Cayley unitary element $u_{[\beta]}$ obtained from the skew-symmetric element $(x+x^{-1})$
is given by $b_0 + b_1 x + b_2x^2 + \cdots + b_{n-1}x^{n-1}$, where $b_0=2a_0-1$ and $b_k=2a_k$,  for $1\leq k \leq n-1$ and
$a_k$'s are as in Theorem \ref{THM-inver-si-o(x)=?}.
\end{theor}
\begin{proof}
Let $x\in G$ be an element with $n=\mathfrak{o}(x)=6t+2$ or $n=\mathfrak{o}(x)=6t+4$ ($\mathfrak{o}(x)\geq 4$) and let $\beta=x+x^{-1}$.
From Theorem \ref{THM-inver-si-o(x)=?} we have that $1+\left(x+x^{-1}\right)$ is invertible and then the Cayley unitary element built
out of $\beta$ is in the way $b_0 +b_1 x+ b_2 x^2 + \cdots + b_{n-1} x^{n-1}$.

Thus,
\begin{align*}
  u_{[\beta]} & = (1-\beta)(1+\beta)^{-1} \\
          & = (1-\left(x+x^{-1}\right))(1+\left(x+x^{-1}\right))^{-1} \\
          &= (1-\left(x+x^{-1}\right))(a_0 + a_1 x + a_2x^2+ \cdots +a_{n-1}x^{n-1}).
\end{align*}

After expand this last product and by expression \eqref{eq:casos}, we obtain that for $1\leq k \leq n-1$

\begin{align*}
  b_0& = - a_{n-1} + a_0 - a_1 =2a_0-1    \\
  b_k& = - a_{k-1} + a_k - a_{k+1}=2a_{k}.
\end{align*}
\end{proof}

In the rest of this section we shall assume that $\mathbf{F}=\mathbb{Q}$, the field of rational numbers, to present some examples of
Cayley unitary elements in specific rational group algebras $\mathbb{Q}G$. In the sequel, we denote by $S_3$ the symmetric group of degree
$3$, by $\mathcal{D}_4$ and $\mathcal{Q}_8$ the dihedral group and quaternion group of order $8$, respectively, and by $C_n$ the
cyclic group of order $n$.

Recall that the set $\mathbf{F}G^{-}$ of skew-symmetric elements of $\mathbf{F}G$ in regard to the oriented classical involution, as
an $\mathbf{F}$-module, is spanned by the set
$$
\mathcal{L}=\underbrace{\{g-g^{-1} : g^2 \neq 1, g\in N \}}_{\mathcal{L}_1}\cup \overbrace{\{g: g^2 =1, g\not\in N\}}^{\mathcal{L}_2}\cup
\underbrace{\{g+g^{-1} : g^2 \neq 1, g\not\in N \}}_{\mathcal{L}_3}.
$$

\begin{rem}\label{rem1}
\begin{enumerate}
\item It is clear that if the involution on $\mathbf{F}G$ is the oriented classical one, $N$ is a subgroup of $\mathcal{U}n(\mathbf{F}G)$. Moreover, by Corollary
      \ref{cor1} $N\not\subset \mathcal{U}n^{C}(\mathbf{F}G)$ if and only if $N$ has even order. Thus, if $y\in N$ has even order $n>1$, by Proposition \ref{Propo-cay.sii.inver}, we can
      ask if $y$ is a product of two elements in $\mathcal{U}n^{C}(\mathbf{F}G)$. In general, the answer is no. In fact, if
      $G=S_3=\langle x, y: x^3=1, y^2=1, x^y=x^{-1}\rangle$ and $\sigma$ is trivial, then $y$ is not a product of two Cayley unitary elements in $\mathbb{Q}S_3$,
      because every skew-symmetric element in this group algebra is of the form $q(x-x^{-1})$, where $q\in \mathbb{Q}$ and thus, $(1+y)-(1-y)q(x-x^{-1})=(1-qx+qyx)(1+y)$. Since $1+y$ is a zero divisor, our claim follows, i.e., $y$ cannot be a Cayley unitary
      element.

\item Note that if $z-z^{-1}\in \mathcal{L}_1$ is such that $1+(z-z^{-1})$ is invertible, then by Vieira and Ribeiro, Theorem \ref{Coef-Fibo}, for any $q\in \mathbb{Q}$, $u_{[q(z-z^{-1})]}$ there exists. For instance, if $\mathfrak{o}(z)=4$, it follows that
      $$
      u_{[q(z-z^{-1})]}=\frac{1}{1+4q^2}\big(1-2qz+4q^2z^2+2qz^3\big), ~q\in \mathbb{Q}.
      $$

\item As we pointed out in Remark \ref{construct_cayley}, if $w\in \mathcal{L}_2$, then for $q\in \mathbb{Q}\setminus \{\pm 1\}$,
      $$
      \big(1+qw\big)^{-1}=\dfrac{1}{1-q^2}-\dfrac{q}{1-q^2}w \quad \text{ and } \quad u_{[qw]}=\frac{1+q^2}{1-q^2}-\frac{2q}{1-q^2}w.
      $$

\item Finally, if the skew-symmetric element $\beta\in \mathcal{L}_3$, i.e., $\beta=x+x^{-1}$ for some $x\in G$ with
      even order $\mathfrak{o}(x)\geq 4$, $\mathfrak{o}(x)\equiv 2,4 \pmod{6}$ and $\sigma(x)=-1$, then $1+(x+x^{-1})$ always is
      invertible, by Proposition \ref{reg-ai}. Thus, the Cayley unitary element $u_{[x+x^{-1}]}$ is obtained from Theorem \ref{THM-Cayleyuni-si-o(x)=?}.
      In particular, if $\mathfrak{o}(x)=4$, it follows that
\begin{align*}
\Big(1+(x+x^{-1})\Big)^{-1} & = a_0+a_1x+a_2x^2+a_3x^3 \\
                            & = \frac{1}{3}+\frac{1}{3}x-\frac{2}{3}x^2+\frac{1}{3}x^3 
\end{align*}

and,

\begin{align}
u_{[x+x^{-1}]} &= \Big(1-(x+x^{-1})\Big)\Big(1+(x+x^{-1})\Big)^{-1} \nonumber  \\     
               &= \Big(1-(x+x^{-1})\Big)\Big(\frac{1}{3}+\frac{1}{3}x-\frac{2}{3}x^2+\frac{1}{3}x^3\Big) \nonumber \\
							&= b_0+b_1x+b_2x^2+b_3x^3 = -\frac{1}{3}+\frac{2}{3}x-\frac{4}{3}x^2+\frac{2}{3}x^3. \label{eq:S3} 
\end{align}

\end{enumerate}

\end{rem}

\begin{examp}
We consider on $S_3$ the non-trivial orientation $\sigma$ given by $\sigma(x)=1$ and $\sigma(y)=-1$. Thus, $x-x^{-1}\in \mathcal{L}_1$,
$y\in \mathcal{L}_2$ and for $q\in \mathbb{Q}$ and $t\in \mathbb{Q}\setminus \left\{\pm 1\right\}$, $q(x-x^{-1})$ and $ty$ are
skew-symmetric elements. In the second case, it follows by Remark \ref{rem1}(3) that $u_{[ty]}=\frac{1+t^2}{1-t^2}-\frac{2t}{1-t^2}y$.
In the former case, using Theorem \ref{Coef-Fibo}, the Cayley unitary element of $\mathbb{Q}S_3$ obtained out of $q(x-x^{-1})$ is given by
$$
u_{[q(x-x^{-1})]}=\frac{1}{1+3q^2}\big(1-q^2+2q(q-1)x+2q(q+1)x^2\big), ~q\in \mathbb{Q}.
$$
\end{examp}

\begin{examp}[Rational quaternion group algebra]
We consider on the quaternion group of order $8$, $\mathcal{Q}_8=\langle x, y: x^4=1, x^2=y^2, y^{-1}xy=x^{-1}\rangle$, different possibilities for
the non-trivial orientation $\sigma$:
\begin{enumerate}
\item If $\sigma(x)=1$ and $\sigma(y)=-1$, $x-x^{-1}\in \mathcal{L}_1$ and $y+y^{-1}, ~xy+(xy)^{-1}\in \mathcal{L}_3$.
      It follows, using the Remark \ref{rem1}(2)(4), that the associated Cayley unitary elements, respectively, are:
$$
u_{[q(x-x^{-1})]}=\frac{1}{1+4q^2}\big(1-2qx+4q^2x^2+2qx^3\big), ~q\in \mathbb{Q}, \text{ and }
$$

\begin{equation}\label{eq:Q8}
u_{[z+z^{-1}]} = -\frac{1}{3}+\frac{2}{3}z-\frac{4}{3}z^2+\frac{2}{3}z^3, \text{ where $z=y$ or $z=xy$. }
\end{equation}

\item In case $\sigma(x)=-1$ and $\sigma(y)=1$, it follows that $y-y^{-1}\in \mathcal{L}_1$ and $x+x^{-1}, ~xy+(xy)^{-1}\in \mathcal{L}_3$
      and thus, the Cayley unitary elements built out of from these skew-symmetric elements are similar to the previous one.

\item Now assume that $\sigma(x)=-1=\sigma(y)$. Since that $x_0=xy$ and $y$ generate $\mathcal{Q}_8$ and $\sigma(x_0)=1$, $\sigma(y)=-1$, then we are in similar conditions as in the first case.
\end{enumerate}

\end{examp}

\begin{examp}[Rational dihedral group algebra]\label{example:D4}Let us consider the finite dihedral group algebra $\mathbb{Q}\mathcal{D}_4$, where $\mathcal{D}_4=\langle x,y: x^4=1=y^2, (xy)^2=1 \rangle $. So considering exactly
the same possibilities for $\sigma$ and following the lines of the previous example, we get that:
\begin{enumerate}
\item For $\sigma(x)=1$ and $\sigma(y)=-1$:
	\begin{align*}
        u_{[q(x-x^{-1})]}&=\frac{1}{1+4q^2}\big(1-2qx+4q^2x^2+2qx^3\big), ~q\in \mathbb{Q},\\[0.2cm]
        u_{[tw]}&=\frac{1+t^2}{1-t^2}-\frac{2t}{1-t^2}w, ~w=y \text{ or } xy ~\text{ and } t\in \mathbb{Q}\setminus\{\pm 1\}.
	\end{align*}
\item For $\sigma(x)=-1$ and $\sigma(y)=1$:
	\begin{align}
u_{[x+x^{-1}]}&= -\frac{1}{3}+\frac{2}{3}x-\frac{4}{3}x^2+\frac{2}{3}x^3,\label{eq:D41}\\[0.2cm]
u_{[t(xy)]}&=\frac{1+t^2}{1-t^2}-\frac{2t}{1-t^2}(xy), \quad u_{[t(x^3y)]}=\frac{1+t^2}{1-t^2}-\frac{2t}{1-t^2}(x^3y),~t\in \mathbb{Q}\setminus\{\pm 1\}. \nonumber
	\end{align}
\item For $\sigma(x)=-1$ and $\sigma(y)=-1$:
\begin{align}
u_{[x+x^{-1}]}&= -\frac{1}{3}+\frac{2}{3}x-\frac{4}{3}x^2+\frac{2}{3}x^3,\label{eq:D42}\\[0.2cm]
		u_{[t(x^2y)]}&=\frac{1+t^2}{1-t^2}-\frac{2t}{1-t^2}(x^2y), \quad u_{[ty]}=\frac{1+t^2}{1-t^2}-\frac{2t}{1-t^2}y,~t\in \mathbb{Q}\setminus\{\pm 1\}.\nonumber
	\end{align}
\end{enumerate}
\end{examp}
\section{Conclusions and future work}
We finish this manuscript given some conclusions and doing some points that can be seen as future work.
\begin{itemize}
	\item We observe from \eqref{eq:S3}, \eqref{eq:Q8}, \eqref{eq:D41} and \eqref{eq:D42} that however $S_3$, $\mathcal{Q}_8$ and $\mathcal{D}_4$ are different groups, the form of the Cayley unitary element $u_{[z+z^{-1}]}$ is similar. In other words, the form of $u_{[z+z^{-1}]}$ depends just on $\mathfrak{o}(z)$. Therefore, to know its representation we can always work as taking $z$ in $C_n$ and just consider the order of $z$ as is shown in Table \ref{Tabla-Cayley-orien}.
	\item From Example \ref{example:D4}, we highlight that the set of Cayley unitary elements obtained studied in this manuscript is bigger than the one obtained in group algebras with the classical orientation given in \cite{VieiraRibeiro:06}. These new units maybe can be considered in the applications that use units from group algebras, for instance in cryptography and coding theory.
	\item In this research, we obtain a formula for the Caylet unitary elements generated from $x+x^{-1}\in \mathcal{L}_1$, but it is an open question to find a formula for $u_{[(q(x+x^{-1})]}$, with $1\neq q\in \mathbf{F}$. More generally, it can be considered to find a closed formula for $u_{[\beta]}$, where $\beta$ is a linear combination of elements of $\mathbf{F}G^{-}$.
\end{itemize}

\begin{table}[H]
\begin{center}
\begin{spacing}{2.5}
\begin{longtable}{|c||c|}
  \hline
  $\mathfrak{o}(z)$ &  $u_{[\beta]}=\left(1-\left(z+z^{\mathfrak{o}(z)-1}\right)\right)\left(1+z+z^{\mathfrak{o}(z)-1}\right)^{-1}$ \\
  \hline
  \hline
  $4$ &  $-\frac{1}{3} + \frac{2}{3}z -\frac{4}{3} z^2+ \frac{2}{3}z^3$ \\
  \hline
  $8$ & $-\frac{5}{3} + \frac{4}{3}z -\frac{2}{3} z^2 - \frac{2}{3}z^3 + \frac{4}{3}z^4 -\frac{2}{3} z^5 - \frac{2}{3}z^6  +\frac{4}{3}z^7$  \\
  \hline
  $10$ & $-\frac{1}{3} + \frac{2}{3}z -\frac{4}{3} z^2 + \frac{2}{3}z^3+ \frac{2}{3}z^4 -\frac{4}{3}z^5 + \frac{2}{3}z^6 + \frac{2}{3}z^7 -\frac{4}{3}z^8 +\frac{2}{3}z^9$ \\
   \hline
  $14$ & $-\frac{5}{3} + \frac{4}{3}z -\frac{2}{3} z^2 - \frac{2}{3}z^3 + \frac{4}{3}z^4 -\frac{2}{3} z^5 - \frac{2}{3}z^6 + \frac{4}{3}z^7 - \cdots +\frac{4}{3}z^{13}$ \\
   \hline
  $16$ & $-\frac{1}{3} + \frac{2}{3}z -\frac{4}{3} z^2+ \frac{2}{3}z^3+ \frac{2}{3}z^4 -\frac{4}{3}z^5 + \frac{2}{3}z^6+ \frac{2}{3}z^7 -\frac{4}{3}z^8+ \cdots +\frac{2}{3}z^{15}$ \\
  \hline
\end{longtable}
\caption{Cayley unitary elements built out of $\beta=(z+z^{-1})$ with $\sigma(z)=-1$.}
\label{Tabla-Cayley-orien}
\end{spacing}
\end{center}
\end{table}

\section*{Acknowledgements}

Y. W. G\'omez-Esp\'indola and A. Holgu\'{i}n-Villa were partially supported by Escuela de Matem\'aticas
- Facultad de Ciencias at Universidad Industrial de Santander. A. Holgu\'{i}n-Villa is member of the research group Álgebra y Combinatoria (ALCOM). J. H. Castillo was partially supported
by Vicerrector\'ia de Investigaciones e Interacci\'on Social at Universidad de Nari\~no  and is member of the research group Algebra, Teor\'ia de N\'umeros y Aplicaciones: ERM (ALTENUA).

\bibliographystyle{plain}

\begin{thebibliography}{10}

\bibitem{BDR:09}
O.~Broche, A.~Dooms, and M.~Ruiz-Mar\'in.
\newblock Unitary units satisfying a group identity.
\newblock {\em Comm. Algebra}, 37(5):1729--1738, 2009.

\bibitem{ChuangLee:95}
C.~L. Chuang and P.~H. Lee.
\newblock Unitary elements in simple artinian rings.
\newblock {\em J. Algebra.}, 176(2):449--459, 1995.

\bibitem{DRM:07}
A.~Dooms and M.~Ruiz~Mar\'in.
\newblock Symmetric units satisfying a group identity.
\newblock {\em J. Algebra}, 308(2):742--750, 2007.

\bibitem{GP:03}
A.~Giambruno and C.~Polcino~Milies.
\newblock Unitary units and skew elements in group algebras.
\newblock {\em Manuscripta Math.}, 111(2):195--209, 2003.

\bibitem{GPS:09i}
A.~Giambruno, C.~Polcino~Milies, and S.~K. Sehgal.
\newblock Group identities on symmetric units.
\newblock {\em J. Algebra}, 322(8):2801--2815, 2009.

\bibitem{GPS:10}
A.~Giambruno, C.~Polcino~Milies, and S.~K. Sehgal.
\newblock Star-group identities and groups of units.
\newblock {\em Arch. Math.}, 95(6):501--508, 2010.

\bibitem{GonPass:01}
J.~Z. Gon{\c{c}}alves and D.~S. Passman.
\newblock Unitary units in group algebras.
\newblock {\em Isr. J. Math.}, 125(1):131--155, 2001.

\bibitem{goodaire2002discrete}
E.~Goodaire and M.~M. Parmenter.
\newblock {\em Discrete mathematics with graph theory}.
\newblock Prentice Hall PTR, second edition, 2002.

\bibitem{holguin2023group}
A.~Holgu{\'\i}n-Villa and J.~H. Castillo.
\newblock Group identities on symmetric units under oriented involutions in
  group algebras.
\newblock {\em Ric. Mat.}, 74(1):1--12, 2025.

\bibitem{Lee:10}
G.~T. Lee.
\newblock {\em Group identities on units and symmetric units of group rings},
  volume~12 of {\em Algebra and Applications}.
\newblock Springer-Verlag London Ltd., London, 2010.

\bibitem{LSS10}
G.~T. Lee, S.~K. Sehgal, and E.~Spinelli.
\newblock Nilpotency of group ring units symmetric with respect to an
  involution.
\newblock {\em J. Pure Appl. Algebra}, 214(9):1592--1597, 2010.

\bibitem{LSS14}
G.~T. Lee, S.~K. Sehgal, and E.~Spinelli.
\newblock Group rings whose unitary units are nilpotent.
\newblock {\em J. Algebra}, 410:343--354, 2014.

\bibitem{LSS18}
G.~T. Lee, S.~K. Sehgal, and E.~Spinelli.
\newblock Bounded engel and solvable unitary units in group rings.
\newblock {\em J. Algebra}, 501:225--232, 2018.

\bibitem{Nov:70}
S.~P. Novikov.
\newblock Algebraic construction and properties of {H}ermitian analogs of
  {$K$}-theory over rings with involution from the viewpoint of {H}amiltonian
  formalism. {A}pplications to differential topology and the theory of
  characteristic classes. {I}. {II}.
\newblock {\em Izv. Akad. Nauk SSSR Ser. Mat.}, 4(2):253--288; ibid. 34 (1970),
  475--500, 1970.

\bibitem{VieiraRibeiro:06}
A.~C. Vieira and V.~R.~T. Da~Silva.
\newblock Unitary units in group algebras and fibonacci sequences.
\newblock {\em J. Algebra Appl.}, 5(2):145--151, 2006.

\end{thebibliography}

\end{document}